%% file: H-H_applications_I.tex
\documentclass[10 pt]{amsart}
\usepackage{amsfonts,amsmath,amsthm,amssymb} 

\usepackage[OT4]{fontenc}
\usepackage[utf8]{inputenc}
\usepackage{enumerate}
\usepackage{esvect}
\usepackage{mathtools}
\usepackage{bm}
\usepackage[svgnames]{xcolor}
\usepackage{float}
\usepackage{tikz}
\DeclareMathOperator{\conv}{conv}
\DeclareMathOperator{\card}{card}
\DeclareMathOperator{\dist}{dist}
\DeclareMathOperator{\Avg}{Avg}
\DeclareMathOperator{\Area}{Area}
\DeclareMathOperator{\Perim}{Perim}
\newcommand{\RR}{\mathbb{R}}
\newcommand{\fb}[1]{\mathbf{#1}}

\newcommand{\Vol}[1]{\mathrm{Vol}({#1})}

\renewcommand{\d}[1]{\,\mathrm{d}{#1}}

\theoremstyle{plain}
\newtheorem{theorem}{Theorem}[section]
\newtheorem{corollary}[theorem]{Corollary}

\theoremstyle{definition}

\theoremstyle{remark}
\newtheorem{remark}{Remark}[section]

\numberwithin{equation}{section}

\begin{document}
\title{Applications of the Hermite-Hadamard inequality}
\author[M. Nowicka]{Monika Nowicka}
\address{Institute of Mathematics and Physics, UTP University of Science and Technology, al. prof. Kaliskiego 7, 85-796 Bydgoszcz, Poland}
\email{monika.nowicka@utp.edu.pl}

\author[A. Witkowski]{Alfred Witkowski}
\email{alfred.witkowski@utp.edu.pl}
\subjclass[2010]{26D15,26B15}
\keywords{Hermite-Hadamard inequality, convex function, polygon, polyhedron, annulus}
\date{21.03.2016}

\begin{abstract}
We show how the recent improvement of the Hermite-Hadamard inequality can be applied to some (not necessarily convex) planar figures and three-dimensional bodies satisfying some kind of regularity.
\end{abstract}
\maketitle
\input{intro}

\input{quad}

\input{fan}

\input{annulus}
\input{polyhedra}
\input{bipyramid}

\input{cube}

\end{document}

%% file: intro.tex
\section{Introduction}
The classical Hermite-Hadamard inequality \cite{Had} states that for a convex function $f\colon[a,b]\to\RR$
\begin{equation}
f\left(\frac{a+b}{2}\right)\leq \frac{1}{b-a}\int_a^b f(t)\d{t}\leq \frac{f(a)+f(b)}{2}.
\label{eq:HH-dim1}
\end{equation}
Due to its simple and elegant form it became a natural object of investigations. Neuman and Bessenyei \cite{Neu,Bes} proved the version for simplices   saying that if $\Delta\subset\RR^n$ is a simplex with barycenter $\mathbf{b}$ and vertices $\mathbf{x}_0,\dots,\mathbf{x}_n$ and $f\colon\Delta\to\RR$ is convex, then
\begin{equation}\label{eq:HH-Bess}
	f(\fb{b})\leq \frac{1}{\Vol{\Delta}}\int_\Delta f(\fb{x})\d{\fb{x}}\leq \frac{f(\fb{x}_0)+\dots +f(\fb{x}_n)}{n+1}.
\end{equation}
The following generalizations for convex function on disk and ball can be found in \cite{DP}:
If $D(O,R)\subset\RR^2$ is a disk and $f\colon D\to\RR$ is convex and continuous, then
$$f(O)\leq \frac{1}{\pi R^2}\iint\limits_{\mathclap{D(O,R)}} f(x,y)\d{x}\d{y}\leq \frac{1}{2\pi R}\int\limits_{\mathclap{\partial D(O,R)}} f(x,y)\d{s}$$
and If $B(O,R)\subset\RR^3$ is a ball and $f\colon B\to\RR$ is convex and continuous, then
$$f(O)\leq \frac{3}{4\pi R^3}\iiint\limits_{\mathclap{B(O,R)}} f(x,y,z)\d{x}\d{y}\d{z}\leq \frac{1}{4\pi R^2}\iint\limits_{\mathclap{\partial B(O,R)}} f(x,y,z)\d{S}.$$
The stronger version of the right-hand side of \eqref{eq:HH-dim1} (\cite[page 140]{PPT})
$$\frac{1}{b-a}\int_a^b f(t)\d{t}\leq \frac{1}{2}\left(f\left(\frac{a+b}{2}\right)+\frac{f(a)+f(b)}{2}\right)$$
also received generalizations for simplices \cite{WW}, disks, 3-balls and regular $n$-gons $P$ \cite{Chen}:
\begin{equation}
\frac{1}{\Vol{\Delta}}\int_\Delta f(\fb{x})\d{\fb{x}}\leq \frac{1}{n+1}f(\fb{b})+\frac{n}{n+1}\frac{f(\fb{x}_0)+\dots +f(\fb{x}_n)}{n+1},
\label{eq:HHstronger}
\end{equation}
$$ \frac{1}{\pi R^2}\iint\limits_{\mathclap{ D(O,R)}} f(x,y)\d{x}\d{y}\leq \frac{1}{3}f(O)+\frac{2}{3}\cdot\frac{1}{2\pi R}\int\limits_{\mathclap{\partial D(O,R)}} f(x,y)\d{s},$$

$$\frac{3}{4\pi R^3}\iiint\limits_{\mathclap{B(O,R)}} f(x,y,z)\d{x}\d{y}\d{z}\leq \frac{1}{4}f(O)+\frac{3}{4}\cdot\frac{1}{4\pi R^2}\iint\limits_{\mathclap{\partial B(O,R)}} f(x,y,z)\d{S},$$

$$\frac{1}{\Area(P)}\iint_P f(x,y)\d{x}\d{y}\leq \frac{1}{3}f(O)+\frac{2}{3\Perim( P)}\int_{\partial P} f(x,y)\d s.$$

In this paper we use the lower and upper estimates for the average of a convex function over a simplex obtained by the authors in \cite{NoWiL,NoWiR} to provide the alternate proof of the above results and to generalize then to   figures and bodies  satisfying some regularity conditions and to broader class of functions.

\section{Definitions and lemmas}
Suppose $\mathbf{x}_0,\dots, \mathbf{x}_n\in\mathbb{R}^n$ are the vertices of a simplex $\Delta\subset\RR^n$.

For a nonempty set $K\subset\{0,\dots,n\}$  we denote by $\Delta_K$ the simplex $\conv \{\mathbf{x}_i: i\in K\}$.\\
For every set $K\subsetneq \{0,\dots,n\}$ we denote by $\Delta^{[K]}$ the simplex with vertices
$$\mathbf{x}_j^{[K]}=\frac{1}{n+1}\sum_{i\in K} \mathbf{x}_i+\frac{n+1-k}{n+1} \mathbf{x}_j,\quad j\in\{0,\dots,n\}\setminus K.$$
We shall denote by $h_\mathbf{a}^\lambda$ the homothety with center $\mathbf{a}$ and scale $\lambda$, i.e. the mapping defined by the formula
$$h_\mathbf{a}^\lambda(\mathbf{x})=\mathbf{a}+\lambda(\mathbf{x}-\mathbf{a}).$$
By $\partial B$ we shall denote the boundary of the set $B$.

\begin{remark}\label{rem:delta^k na plaszczyznie}
	In the plane the simplices $\Delta^{[K]}$ are: the triangle (if $K=\emptyset$), intersection of the triangle and a line parallel to one of its sides and passing through its barycenter (if $K$ has one element) and the barycenter itself if $K$ has two elements.\\ 
In case of three dimensions we have respectively the tetrahedron, triangles parallel to its faces, segments parallel to its edges, all of them having the same barycenter.\\
Note that the simplices $\Delta^{[K]}$ can be obtained by applying homotheties to the faces of $\Delta$. The details are explained in \cite{NoWiL}.
\end{remark}

If $U\subset \RR^k$ and    $f\colon U\to\RR$ is a Riemann integrable function, then by 
$$\Avg(f,U)=\frac{1}{\Vol{U}} \int_U f(\mathbf{x})\d{\mathbf{x}}$$
we shall denote its average value over $U$.
For simplicity of notation if $A,B,\dots,K$ are points and $U=\conv\{A,B,\dots,K\}$ we shall write $\Avg(f,AB\dots K)$.

The following results provide the main tools for our investigations:

\begin{theorem}[\cite{NoWiR}]\label{th:R}
Suppose $f\colon\Delta\to\RR$ is a convex function and $K,L\subset \{0,\dots,n\}$ are disjoint, nonempty sets. Then
$$	\Avg(f,\Delta_{K\cup L})\leq \frac{\card K}{\card K\cup L}\cdot \Avg(f,\Delta_{K}) + 
\frac{\card L}{\card K\cup L}\cdot  \Avg(f,\Delta_{L}).$$
\end{theorem}

\begin{theorem}[\cite{NoWiL}]\label{th:L}
If $K\subset L$ are proper subsets of $\{0,\dots,n\}$ and $f\colon\Delta\to\RR$ is a convex function, then
$$f(\mathbf{b})\leq \Avg(f,\Delta^{[L]})\leq \Avg(f,\Delta^{[K]})\leq \Avg(f,\Delta).$$
	
\end{theorem}
The above theorems were proven by Chen \cite{Chen} in case $\Delta$ is a triangle.

In the sequel we shall apply both theorems to some planar and 3-dimensional bodies.

%% file: quad.tex
\section{Quadrilateral}
Besseneyi in \cite{Bes} proved that if $ABCD$ is a parallelogram  and $f$ is a convex function, then $\Avg(f,ABCD)\leq \frac{1}{4}(f(A)+f(B)+f(C)+f(D))$.\\
We will try to generalize and improve this result.

Consider a quadrilateral $ABCD$  such that the segment $AC$ divides its area evenly (see Figure \ref{pic:Quadrilateral with equal halves}a).
\begin{figure}[h]%
\begin{center}
\includegraphics{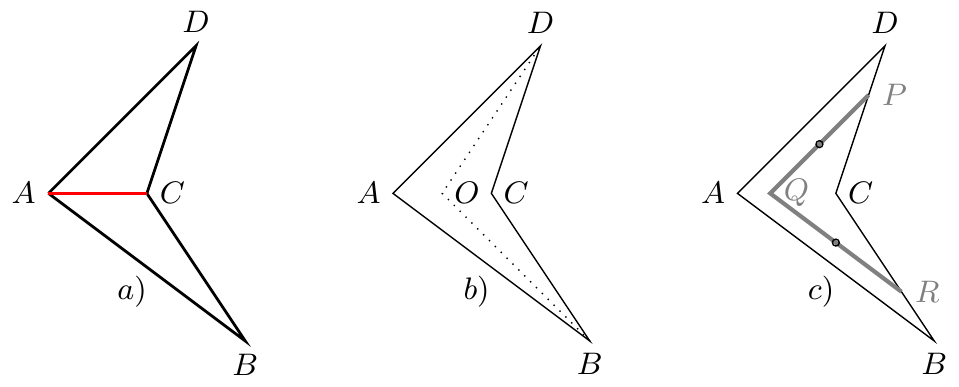}
\caption{Quadrilateral with equal halves}%
\label{pic:Quadrilateral with equal halves}%

\end{center}
\end{figure}

 We can apply Theorem \ref{th:R} to both triangles $ABC$ and $ADC$ to obtain
\begin{align}
	\Avg(f,ABC)&\leq \frac{1}{3}f(B)+\frac{2}{3}\Avg(f,AC) \label{neq:quadrilateral<ABC},	\\
	\Avg(f,ACD)&\leq \frac{1}{3}f(D)+\frac{2}{3}\Avg(f,AC) \label{neq:quadrilateral<ADC},	
\end{align}
which yields
\begin{align}\label{neq:quadrilateral<}
\Avg(f,ABCD)&\leq\frac{1}{3}\left(\frac{f(B)+f(D)}{2}+2\Avg(f,{AC})\right)\\
	&\leq\frac{1}{3}\left(\frac{f(B)+f(D)}{2}+f(A)+f(C)\right)\notag.
\end{align}

By adding a midpoint $O$ of the segment $AC$ we can get another upper bound (see Figure \ref{pic:Quadrilateral with equal halves}b):
\begin{align*}
	\Avg(f,AOB)&\leq\frac{1}{3}f(O)+\frac{2}{3}\Avg(f,AB)\\	
	\Avg(f,BOC)&\leq\frac{1}{3}f(O)+\frac{2}{3}\Avg(f,BC)\\	
	\Avg(f,COD)&\leq\frac{1}{3}f(O)+\frac{2}{3}\Avg(f,CD)\\	
	\Avg(f,DOA)&\leq\frac{1}{3}f(O)+\frac{2}{3}\Avg(f,DA)	
\end{align*}
and since the four triangles have the same area this produces
\begin{align}
	{\Avg(f,ABCD)}&{\leq\frac{f(O)}{3}+\frac{2}{3}\frac{\Avg(f,AB)+\Avg(f,BC)+\Avg(f,CD)+\Avg(f,DA)}{4}}	\label{neq:ABCDO1}\\
	&	{\leq\frac{f(O)}{3}+\frac{2}{3}\frac{f(A)+f(B)+f(C)+f(D)}{4}}.\label{neq:ABCDO2}
\end{align}
Thus we have proven the following
\begin{theorem}\label{th:quadrilateral1}
	Let $ABCD$ be a quadrilateral such that the segment $AC$ divides it into two triangles of equal area and $O$ be the midpoint of $AC$. If $f\colon ABCD\to\RR$ is such that its restrictions to triangles $ABC$ and $ACD$ are convex, then the inequalities \eqref{neq:quadrilateral<}, \eqref{neq:ABCDO1} and \eqref{neq:ABCDO2} hold.
\end{theorem}
 To obtain the lower bound we apply Theorem \ref{th:L} to both triangles $ABC$ and $ADC$. By Remark \ref{rem:delta^k na plaszczyznie} we have four reasonable choices for each triangle, so we can produce 16 inequalities. An example is shown on the Figure \ref{pic:Quadrilateral with equal halves}c: the segments $PQ=h_C^{2/3}(DA)$ and $QR=h_C^{2/3}(AB)$ pass through the barycenters of both triangles and therefore $\Avg(f,PQ)\leq \Avg(f,ACD)$ and $ \Avg(f,QR)\leq \Avg(f,ABC)$, which leads to
$$\frac{\Avg(f,PQ)+\Avg(f,QR)}{2}\leq \Avg(f,ABCD).$$
The reader will easily find two other pairs of segments for which Theorem \ref{th:L} can be applied.

A parallelogram offers more opportunities: firstly, we obtain inequalities \eqref{neq:quadrilateral<ABC} and \eqref{neq:quadrilateral<ADC} with $BD$ and $AC$ swapped thus obtaining an improvement of Bessenyei's result
\begin{theorem}\label{th:quadrilateral2}
	Let ABCD be a parallelogram with center $O$ and $f\colon ABCD\to\RR$ be such that its restrictions to triangles $AOB$, $BOC$, $COD$ and $DOA$ are convex, then the inequalities \eqref{neq:ABCDO1} and \eqref{neq:ABCDO2} hold and additionally
\begin{align*}
	\Avg(f,ABCD)&\leq\frac{1}{3}\min\left\{{\textstyle\frac{f(B)+f(D)}{2}+2\Avg(f,{AC}),\frac{f(A)+f(C)}{2}+2\Avg(f,{BD})}\right\}	\\
	&	\leq \frac{1}{3}\min\left\{{\textstyle\frac{f(B)+f(D)}{2}+f(A)+f(C),\frac{f(A)+f(C)}{2}+f(B)+f(C)}\right\}\\
	&\leq \tfrac{f(A)+f(B)+f(C)+f(D)}{4}.
\end{align*}
\end{theorem}
The fact that the parallelogram can be divided into four triangles of equal area opens new opportunities. For example we can apply Theorem \ref{th:R} to $AOB$ (and then cyclically to others) as follows:
$$\Avg(f,AOB)\leq\frac{1}{3}f(A)+\frac{2}{3}\Avg(f,OB).$$
Summing and taking into account that the point $O$ halves both diagonals we get
\begin{theorem}
	Under the assumption of Theorem \ref{th:quadrilateral2}
	$$\Avg(f,ABCD)\leq \frac{1}{3}\frac{f(A)+f(B)+f(C)+f(D)}{4}+\frac{2}{3}\frac{\Avg(f,AC)+\Avg(f,BD)}{2}.$$
	\end{theorem}
The reader will find more estimates applicable to parallelograms and rhombus in Section \ref{sec:ngons} devoted to polygons.

As above using Theorem \ref{th:L} we can produce 64 different lower bounds. Figure \ref{fig:parallelograms} illustrates two,  probably  the most spectacular, inequalities:
\begin{theorem}
Under the assumption of Theorem \ref{th:quadrilateral2} let $A'B'C'D'=h_O^{2/3}(ABCD)$, $KL=h_D^{2/3}(AO), LM=h_D^{2/3}(OC), PQ=h_B^{2/3}(AO), QR=h_B^{2/3}(OC)$. Then the following inequalities hold (see Figure \ref{fig:parallelograms})

$$\frac{\Avg(f,KL)+\Avg(f,LM)+\Avg(f,PQ)+\Avg(f,QR)}{4}\leq \Avg(f,ABCD),$$
$$\frac{\Avg(f,A'B')+\Avg(f,B'C')+\Avg(f,C'D')+\Avg(f,D'A')}{4} \leq \Avg(f,ABCD).$$
\end{theorem}
\begin{figure}[h]
\begin{center}
\includegraphics{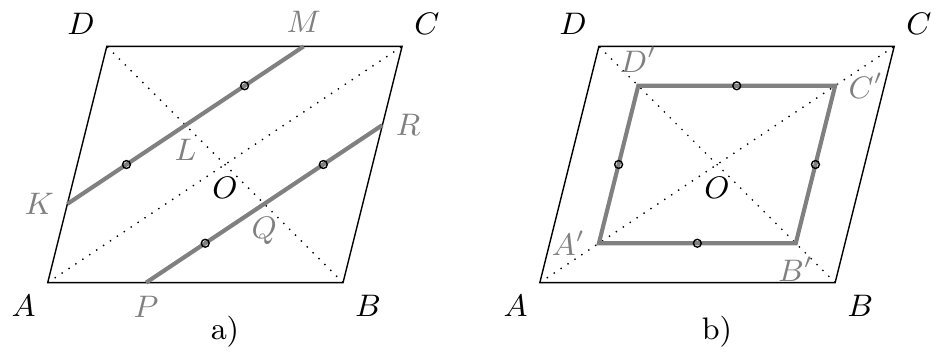}
\caption{Lower bounds for parallelogram}%
\label{fig:parallelograms}%
\end{center}
\end{figure}
\begin{remark}\label{rem:about f(O)}
	If $f$ is convex on the parallelogram, then obviously is convex an all four triangles. The converse does not hold. In both cases one can easily conclude the inequality
	$$\frac{1}{4}\sum_{k=1}^4 f(O_k)\leq\Avg(f,ABCD),$$
	where $O_k$'s are the barycenters of the triangles. This inequality in case of convex $f$ yields $f(O)\leq \Avg(f,ABCD)$. In the case of convexity on triangles only this may not be true (consider the quadrilateral $\{(x,y): |x|+|y|=1\}$ and $f(x,y)=-|y|$).
\end{remark}

%% file: fan.tex
\section{Fans and $n$-gons} \label{sec:ngons}

We shall call \textit{nice}  a star-shaped polygon $P$ with vertices $A_0,\dots,A_{n-1}$ satisfying the following condition: there exists a point $O$ called \textit{center} in the kernel of $P$ such that all triangles $OA_kA_{k+1},\ k=0,\dots,n-1$ are of the same area. If additionally all segments $A_kA_{k+1}$ are of the same length, then we shall call it \textit{very nice}.\\
Regular $n$-gons are very nice, but the class of nice and very nice polygons is much broader.

\begin{theorem}\label{thm:poly}
	If $P$ is a nice polygon  and $f:P\to\RR$ is convex on every triangle formed by its center $O$ and two consecutive vertices, then
	\begin{align}
	\Avg({f,P})&\leq \frac{1}{3}f(O)+\frac{2}{3}\cdot\frac{1}{n}\sum_{k=0}^{n-1}\Avg(f,A_kA_{k+1})\label{neq:poly1},
	\\
	\Avg({f,P})&\leq \frac{1}{3}\cdot\frac{1}{n}\sum_{k=0}^{n-1}f(A_k)+\frac{2}{3}\cdot\frac{1}{n}\sum_{k=0}^{n-1}\Avg(f,OA_k),
	\label{neq:poly2}\\
	\Avg({f,P})&\leq \frac{1}{3}f(O)+\frac{2}{3}\cdot\frac{1}{n}\sum_{k=0}^{n-1}f(A_k).
	\label{neq:poly}
	\end{align}
	If additionally $A_k'=h_O^{2/3}(A_k)$ for $k=0,\dots,n-1$, then
	\begin{equation}
	\frac{1}{n}\sum_{k=0}^{n-1}\Avg(f,A_k'A_{k+1}')\leq \Avg(f,P)
	\label{neq:polyinv}
	\end{equation}
	(see Figure \ref{fig:HermiteHadamardInequalityInNGon}).
\end{theorem}
\begin{proof}
	To obtain \eqref{neq:poly1} apply Theorem \ref{th:R} to vertex $O$ and side $A_kA_{k+1}$, then add up the inequalities. Similarly, for \eqref{neq:poly2} use vertex $A_k$ and side $OA_{k+1}$. The inequality \eqref{neq:poly} can be obtained from \eqref{neq:poly1} or \eqref{neq:poly2} by applying standard Hermite-Hadamard inequalities.\\
	Finally \eqref{neq:polyinv} is  consequence of Theorem \ref{th:L}.
\end{proof}
\begin{remark}
	Every triangle is a nice polygon with its barycenter as $O$.
\end{remark}
\begin{remark}
	In case of a very nice polygon, the inequalities \eqref{neq:poly1} and \eqref{neq:polyinv} can be written as
	\begin{equation*}
	\Avg(f,\partial h_O^{2/3}(P))\leq \Avg(f,P)\leq \frac{1}{3}f(O)+\frac{2}{3}\Avg(f,\partial P).
	\end{equation*}
\end{remark}
Suppose $n$ is even. Then we can group the triangles in pairs to get
\begin{align*}
\Avg(f,OA_kA_{k+1})	&	\leq \frac{1}{3}f(A_{k+1})+\frac{2}{3}\Avg(f,OA_k)\\
\Avg(f,OA_{k-1}A_{k})	&	\leq \frac{1}{3}f(A_{k-1})+\frac{2}{3}\Avg(f,OA_k).
\end{align*}
This shows that the following result holds true.
\begin{theorem}\label{thm:poly2n}
	Under assumptions of Theorem \ref{thm:poly} if the number of vertices of $P$ is even, then
	\begin{align*}
	\Avg(f,P)	&	\leq\frac{1}{3}\frac{1}{n/2}\sum_{k\text{ odd}}f(A_k)+\frac{2}{3}\frac{1}{n/2}\sum_{k\text{ even}}\Avg(f,OA_k),\\
	\Avg(f,P)	&	\leq\frac{1}{3}\frac{1}{n/2}\sum_{k\text{ even}}f(A_k)+\frac{2}{3}\frac{1}{n/2}\sum_{k\text{ odd}}\Avg(f,OA_k),
	\end{align*}
	(see Figure \ref{fig:HermiteHadamardInequalityInNGon}).
\end{theorem}

\begin{figure}[htbp]
\begin{center}
\includegraphics{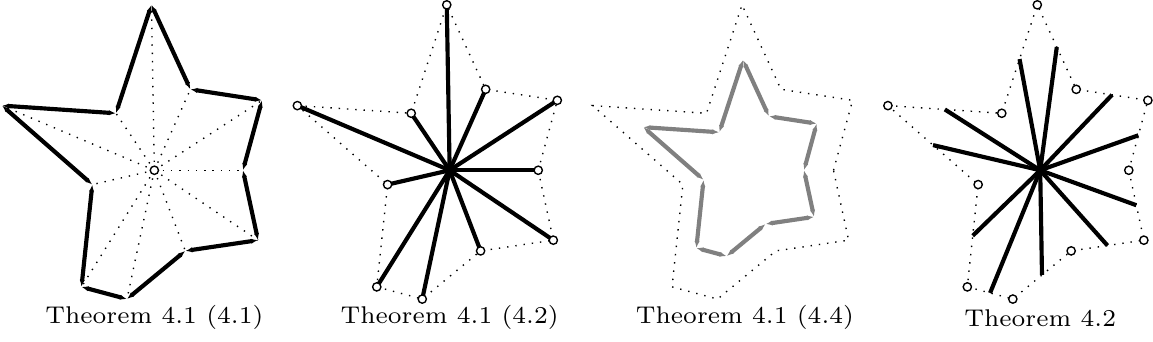}
	\caption{Hermite-Hadamard inequalities for $n$-gon}
	\label{fig:HermiteHadamardInequalityInNGon}
\end{center}
\end{figure}
\begin{remark}	Note that every nice $n$-gon can be considered a nice $2n$-gon by adding the midpoints of its sides to the set of vertices (see Figure \ref{fig:HermiteHadamardInequalityInNGon}).
	
	Remark \ref{rem:about f(O)} remains valid also in this case.
\end{remark}

A fan is a polygon with  vertices $O,A_1,\dots,A_n$ such that all triangles $OA_kA_{k+1}$, $k=1,\dots,n-1$ are of the same orientation and $\sum_{k=1}^{n-1} \angle A_kOA_{k+1}< 2\pi$. For nice fans  we obtain similar results as for nice $n$-gons. We encourage the reader to formulate an equivalent of Theorems \ref{thm:poly} and \ref{thm:poly2n}.

%% file: annulus.tex
\section{Annulus}
In \cite{Chen} the following version of  Hermite-Hadamard inequality can be found
\begin{theorem}[\cite{Chen}, Th. 2.2]\label{thm:annulus Chen}
Let $U$ be a convex subset of a plane, and $D\subset U$ be an annulus with radii $r<R$ and $C(s)$ denotes a co-centric circle with radius $s$, then for a convex function $f\colon U\to\RR$ hold
$$\Avg\left(f,C\left(\tfrac{2(r^2+rR+R^2)}{3(r+R)}\right)\right)\leq \Avg(f,D)$$
and 
$$\Avg(f,D)\leq \frac{2r+R}{3(r+R)}\Avg(f,C(r))+\frac{r+2R}{3(r+R)}\Avg(f,C(R)).$$
\end{theorem}
We shall improve this result. 
For $n>4$ let $A_k^n, k=0,\dots,n-1$ be the vertices of a regular $n$-gon inscribed in $C(R)$ and $B_k^n, k=0,\dots,n-1$ be the vertices of a regular $n$-gon inscribed in $C(r)$ and rotated anticlockwise by $\frac{\pi}{n}$. The two polygons bound the area $D_n$, and divide it into $2n$ isosceles triangles $\mathcal{K}_k^n=A_k^nA_{k+1}^nB_k^n$ and $\mathcal{L}_k^n=B_k^nB_{k+1}^nA_{k+1}^n$. We have

\begin{gather}
\Area{\mathcal{K}_k^n}=R\sin\frac{\pi}{n}\left(R\cos\frac{\pi}{n}-r \right),\quad \Area{\mathcal{L}_k^n}=r\sin\frac{\pi}{n}\left(R-r\cos\frac{\pi}{n}\right),\label{eq:annulus area1}\\
\Area{D_n}=\sum_{k=0}^{n-1}\left(\Area{\mathcal{K}_k^n}+\Area{\mathcal{L}_k^n}\right) =n\sin\frac{\pi}{n}\cos\frac{\pi}{n}(R^2-r^2).\label{eq:annulus area2}
\end{gather}
Denote by $K_k^n,L_k^n$ the barycenters of $\mathcal{K}_k^n$ and $\mathcal{L}_k^n$. Applying the Hermite-Hadamard inequality, equations \eqref{eq:annulus area1} and \eqref{eq:annulus area2} we obtain
\begin{align}\label{eq:average of D_n}
	\Avg(f,D_n)&=\frac{1}{\Area{D_n}}\sum_{k=0}^{n-1}\left(\int_{\mathcal{K}_k^n} f(\mathbf{x})\d{\mathbf{x}}+\int_{\mathcal{L}_k^n} f(\mathbf{x})\d{\mathbf{x}}\right)	\notag\\
&=\sum_{k=0}^{n-1}\left(\frac{\Area{\mathcal{K}_k^n}}{\Area{D_n}}\Avg{(f,\mathcal{K}_k^n)}+\frac{\Area{\mathcal{L}_k^n}}{\Area{D_n}}\Avg{(f,\mathcal{L}_k^n)}\right)\notag	\\
	&\geq \frac{R\left(R\cos\frac{\pi}{n}-r \right)}
	{\cos\frac{\pi}{n}(R^2-r^2)}
	\frac{1}{n}\sum_{k=0}^{n-1} f(K_k^n)
	+\frac{r\left(R-r\cos\frac{\pi}{n}\right)}
	{\cos\frac{\pi}{n}(R^2-r^2)}
	\frac{1}{n}\sum_{k=0}^{n-1} f(L_k^n).
\end{align}
As $n$ tends to infinity the two regular polygons with vertices $K_k^n$ and $L_k^n$ respectively approach the circles of radii $\frac{1}{3}(2R+r)$ and  $\frac{1}{3}(R+2r)$, and the arithmetic means in \eqref{eq:average of D_n} tend to averages of $f$ over these circles, so we have proven the following fact.

\begin{theorem}\label{the:annulus usL}
	Under the assumptions of Theorem \ref{thm:annulus Chen} the inequality
	$$\frac{R}{r+R}\Avg\left(f,C\left(\frac{r+2R}{3}\right)\right)+\frac{r}{r+R}\Avg\left(f,C\left(\frac{2r+R}{3}\right)\right)\leq\Avg(f,D)$$ holds.
\end{theorem}
Similar reasoning and the strengthened version of the Hermite-Hadamard inequality \eqref{eq:HHstronger}
applied to $\mathcal{K}_k^n$ and $\mathcal{L}_k^n$ produce a better right bound.
\begin{theorem}\label{the:annulus usR}
	Under the assumptions of Theorem \ref{thm:annulus Chen} the inequality
	\begin{align*}
	\Avg(f,D)	&	\leq\frac{1}{3}\left[\frac{R}{r+R}\Avg\left(f,C\left(\frac{r+2R}{3}\right)\right)+\frac{r}{r+R}\Avg\left(f,C\left(\frac{2r+R}{3}\right)\right)\right]\\
		&	\phantom{\leq}+\frac{2}{3}\left[\frac{r+2R}{3(r+R)}\Avg(f,C(R))+\frac{2r+R}{3(r+R)}\Avg(f,C(r))\right]
	\end{align*} holds.
	
\end{theorem}
\begin{remark}
	Suppose the function $f$ is such there exist a point $O\in U$ and   half-lines starting from $O$ such that $f$ is convex in each sector bounded by them. Then, if $O$ is the center of $D$, the inequalities in \eqref{eq:average of D_n} are  valid for all triangles except these intersecting with the sectors' boundaries. Thus they can be neglected as $n$ tends to infinity, and the Theorems \ref{the:annulus usL} and \ref{the:annulus usR} remain valid for $f$.
\end{remark}

%% file: polyhedra.tex
\section{Platonic bodies and related polytopes}\label{sec:Platonic}
Let $B\subset\RR^3$ be a platonic body  inscribed in a sphere with center $O$. Define the following sets:
\begin{itemize}
	
	\item $\mathcal{S}$ - set of segments joining $O$ with vertices of $B$
	\item $\mathcal{O}$ - set of segments joining $O$ with centers of faces
	\item $\mathcal{E}$ - set of edges of $B$
	\item $\mathcal{D}$ - set of segments joining centers of faces with their vertices.  
\end{itemize}
 
\begin{theorem}\label{thm:platonic}
	Let $f\colon B\to\RR$ be a function such that its restriction to every pyramid formed by a face as a base and $O$ as its apex is convex. Then with the above notation the following inequalities hold:
\begin{align}
	\Avg(f,B)&\leq \frac{1}{4}f(O)+\frac{3}{4}\Avg(f,\partial B),	\label{eq:Plato_1}\\
	\Avg(f,B)&\leq \frac{1}{2}\Avg(f,\mathcal{O})+\frac{1}{2}\Avg(f,\mathcal{E}),	\label{eq:Plato_2}\\
	\Avg(f,B)&\leq \frac{1}{2}\Avg(f,\mathcal{S})+\frac{1}{2}\Avg(f,\mathcal{D}),	\label{eq:Plato_3}
\intertext{and}
	\Avg(f,\partial h_O^{3/4}(B))&\leq \Avg(f,B)\label{eq:Plato_4}.
\end{align}
\end{theorem}
\begin{proof}
	Let $F$ be a face of $B$ with vertices $A_0,\dots,A_{n-1}$ and $O'$ be its circumcenter. Split the pyramid $FO$ into simplices $OO'A_kA_{k+1}$. Applying Theorem \ref{th:R} we obtain
	$$\Avg{(f,OO'A_kA_{k+1})}\leq \frac{1}{4}f(O)+\frac{3}{4}\Avg(f,O'A_kA_{k+1}).$$
	Summing over $k$ and $F$ we obtain \eqref{eq:Plato_1}.
	The inequalities
	$$\Avg(f,OO'A_kA_{k+1})\leq\frac{1}{2}\Avg(f,OO')+\frac{1}{2}\Avg(f,A_kA_{k+1})$$
	lead to \eqref{eq:Plato_2}, while
	$$\Avg{(f,OO'A_kA_{k+1})}\leq\frac{1}{2}\Avg(f,OA_k)+\frac{1}{2}\Avg(f,O'A_{k+1})$$
	give \eqref{eq:Plato_3}.
	Finally Theorem \ref{th:L} leads to the inequalities 
	$$\Avg{(f,h_O^{3/4}(O'A_kA_{k+1}))}\leq \Avg{(f,OO'A_kA_{k+1})}$$
	that finally yield \eqref{eq:Plato_4}. 
\end{proof}

Based on a platonic body $B$ we can built a new polytope $B^*$ in the following way: on every face $F$ of $B$ we build or excavate a regular pyramid  of the same height with apex $O_F$. If we denote by 
 \begin{itemize}
	\item $\mathcal{S}$ - set of segments joining $O$ with vertices of $B$,
	\item $\mathcal{O}^*$ - set of segments joining $O$ with $O_F$'s,
	\item $\mathcal{E}$ - set of edges of $B$,
	\item $\mathcal{D}^*$ - set of segments joining $O_F$'s of the pyramids with  vertices of $F$, 
\end{itemize}
then the same reasoning as above shows that the next theorem is valid.
\begin{theorem}
	Let $f\colon B^*\to\RR$ be a function such that its restriction to every tetrahedron formed by $O$, $O_F$ and two adjacent vertices of $F$ . Then with the above notation the following inequalities hold:
\begin{align*}
	\Avg(f,B^*)&\leq \frac{1}{4}f(O)+\frac{3}{4}\Avg(f,\partial B^*),	\\
	\Avg(f,B^*)&\leq \frac{1}{2}\Avg(f,\mathcal{O}^*)+\frac{1}{2}\Avg(f,\mathcal{E}),	\\
	\Avg(f,B^*)&\leq \frac{1}{2}\Avg(f,\mathcal{S})+\frac{1}{2}\Avg(f,\mathcal{D}^*),	
\intertext{and}
	\Avg(f,\partial h_O^{3/4}(B^*))&\leq \Avg(f,B^*).
\end{align*}
\end{theorem}

%% file: bipyramid.tex
\section{Dipyramid and Dicone}

Let $P$ be a regular, convex $n$-gon with vertices $A_0,\dots,A_{n-1}$. Suppose $X$ is a point on the line $l$ perpendicular to the  plane containing $P$ and passing through its center.
The set $\mathcal{U}_X=\bigcup_{k=0}^{n-1} XA_kA_{k+1}$ will be called a \textit{Chinese umbrella with vertex $X$}. The umbrella's scaffold will be denoted by $\mathcal{S}_X=\bigcup_{k=0}^{n-1} XA_k$. \\
By a \textit{dipyramid} with vertices $O_0, O_1$ we mean the body $\mathcal{D}$ bounded by two Chinese umbrellas  $\mathcal{U}_{O_0}$ and $\mathcal{U}_{O_1}$. The dipyramid may be convex or not, depending on the position of its vertices with respect to the plane of the polygon. 

The following result is a consequence of Theorems \ref{th:R} and \ref{th:L}.
\begin{theorem}\label{thm:dipyramid1}
	Let $\mathcal{D}$ be a dipyramid with vertices $O_0$ and $O_1$, and let $f:\mathcal{D}\to\RR$ be a function that is convex on every simplex $O_0O_1A_kA_{k+1},\ k=0,\dots,n-1$. Then
	\begin{align}
	\Avg(f,\mathcal{D})	&\leq \tfrac{1}{4}f(O_0)+\tfrac{3}{4}\Avg(f,\mathcal{U}_{O_1})	\label{eq:dipyramidO_0},\\
	\Avg(f,\mathcal{D})	&\leq \tfrac{1}{4}f(O_1)+\tfrac{3}{4}\Avg(f,\mathcal{U}_{O_0})	\label{eq:dipyramidO_1},\\
	\Avg(f,h_{O_0}^{3/4}(\mathcal{U}_{O_1}))&\leq \Avg(f,\mathcal{D}) \label{eq:dipyramidO_0_inv},\\
	\Avg(f,h_{O_1}^{3/4}(\mathcal{U}_{O_0}))&\leq \Avg(f,\mathcal{D}) \label{eq:dipyramidO_1_inv},\\
	\Avg(f,\mathcal{D})	&\leq \tfrac{1}{2}\left(\Avg(f,O_0O_1)+\Avg(f,\partial P)\right)	\label{eq:dipyramidO_0O_1},\\
	\Avg(f,\mathcal{D})	&\leq \tfrac{1}{2}\left(\Avg(f,\mathcal{S}_{O_0})+\Avg(f,\mathcal{S}_{O_1})\right),\label{eq:dipyramid_sticks}	
	\end{align}
	(see Figure \ref{fig:HermiteHadamardInewualitiesForDipyramid}).
\end{theorem}
\begin{proof}
	Grouping the vertices of the simplex $O_0O_1A_kA_{k+1}$ into  $\{O_0\},\{O_1A_kA_{k+1}\}$ one gets the inequalities \eqref{eq:dipyramidO_0} and \eqref{eq:dipyramidO_0_inv}. Similar split $\{O_1\},\{O_0A_kA_{k+1}\}$ gives \eqref{eq:dipyramidO_1} and \eqref{eq:dipyramidO_1_inv}. \\
	Inequalities \eqref{eq:dipyramidO_0O_1} and \eqref{eq:dipyramid_sticks} follow by grouping them into $\{O_0O_1\}$, $\{A_kA_{k+1}\}$ and $\{O_0A_k\}$, $\{O_1A_{k+1}\}$ respectively.
\end{proof}
\begin{figure}[htbp]
\begin{center}
	\includegraphics{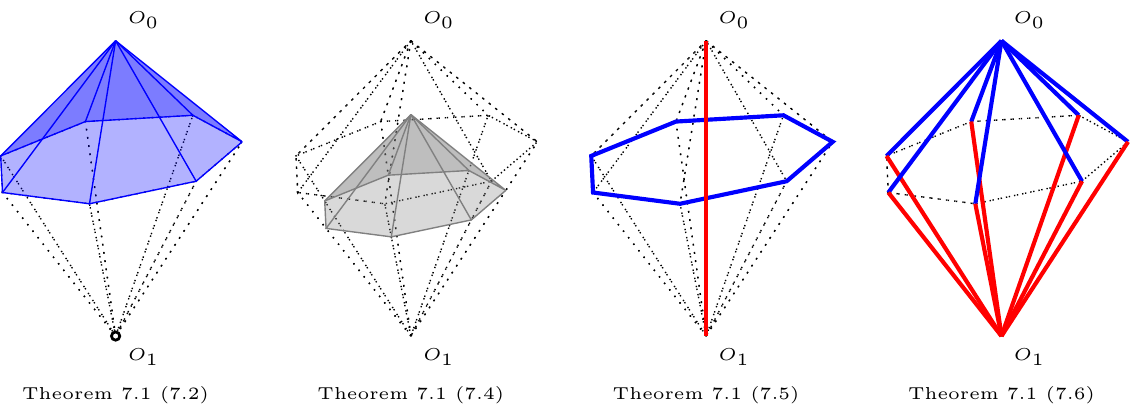}
	\caption{Hermite-Hadamard inequalities for dipyramid}
	\label{fig:HermiteHadamardInewualitiesForDipyramid}
\end{center}
\end{figure}

Denote by $\eta_i,\ i=0,1$ the angle between the line $l$ and the plane of a side of $\mathcal{U}_{O_i}$ and for $0<s<1$ let $O_s=(1-s)O_0+sO_1$. The umbrella $\mathcal{U}_{O_s}$ splits the dipyramid $\mathcal{D}$ into two dipyramids $\mathcal{D}_s^0$ and $\mathcal{D}_s^1$ with vertices $O_0,O_s$ and $O_1,O_s$ respectively. It is clear, that
\begin{equation}
\Vol{\mathcal{D}_s^0}=s\Vol{\mathcal{D}}\quad\text{and}\quad \Vol{\mathcal{D}_s^1}=(1-s)\Vol{\mathcal{D}}.
\label{eq:dipyramid_volume_split}
\end{equation}
Note also that 
$$\frac{\Area(\mathcal{U}_{O_0})}{\Area(\mathcal{U}_{O_1})}=\frac{\sin \eta_1}{\sin \eta_0}$$
which gives
\begin{equation}
\Area(\partial\mathcal{D})=\frac{\sin\eta_0+\sin\eta_1}{\sin\eta_1}\Area(\mathcal{U}_{O_0})=\frac{\sin\eta_0+\sin\eta_1}{\sin\eta_0}\Area(\mathcal{U}_{O_1}).
\label{eq:dipyramid_side_areas}
\end{equation}
Now we are ready to generalize the results of Theorem \ref{thm:dipyramid1}.
\begin{theorem}\label{thm:dipyramid2}
	Under the assumptions of Theorem \ref{thm:dipyramid1} for all $0<s<1$ the inequalities 
\begin{equation}
	\Avg(f,\mathcal{D})\leq \tfrac{1}{4}f(O_s)+\tfrac{3}{4}\left(s\Avg(f,\mathcal{U}_{O_0})+(1-s)\Avg(f,\mathcal{U}_{O_1})\right),	\label{neq:dipyramidO_s}
\end{equation}
\begin{equation}
s\Avg(f,h_{O_s}^{3/4}(\mathcal{U}_{O_0}))+(1-s)\Avg(f,h_{O_s}^{3/4}(\mathcal{U}_{O_1}))\leq \Avg(f,\mathcal{D}),
\label{neq:dipyramidO_s_inv1}
\end{equation}
\begin{equation}
\Avg(f,\partial h_{O_s}^{3/4}(\mathcal{D}))\leq \frac{\sin\eta_1\Avg(f,\mathcal{D}_s^0)}{\sin\eta_0+\sin\eta_1}+\frac{\sin\eta_0\Avg(f,\mathcal{D}_s^1)}{\sin\eta_0+\sin\eta_1}
\label{neq:dipyramidO_s_inv2}
\end{equation}	
are valid.
\end{theorem}
\begin{proof}
	To prove \eqref{neq:dipyramidO_s} we apply inequality \eqref{eq:dipyramidO_0} to $\mathcal{D}_s^0$ and to $\mathcal{D}_s^1$ and obtain
	\begin{align*}
		\Avg(f,\mathcal{D}_s^0)	&\leq \tfrac{1}{4}f(O_s)+\tfrac{3}{4}\Avg(f,\mathcal{U}_{O_0}),	\\
		\Avg(f,\mathcal{D}_s^1)	&\leq \tfrac{1}{4}f(O_s)+\tfrac{3}{4}\Avg(f,\mathcal{U}_{O_1}).	
	\end{align*}
	Now we multiply the first inequality by $s$, the second by $1-s$ and add up both inequalities taking into account equalities \eqref{eq:dipyramid_volume_split}.
	
	The proof of \eqref{neq:dipyramidO_s_inv1} is similar but uses \eqref{eq:dipyramidO_0_inv} and \eqref{eq:dipyramidO_1_inv}.
	
	And finally from \eqref{eq:dipyramidO_0_inv}, \eqref{eq:dipyramidO_1_inv} and \eqref{eq:dipyramid_side_areas} (which remains valid for homothetic images also) it follows that
	$$\frac{\sin\eta_0+\sin\eta_1}{\sin\eta_{1-i}}\frac{1}{\Area(\partial h_{O_s}^{3/4}(\mathcal{D}))}\int_{h_{O_s}^{3/4}(\mathcal{U}_{O_i})}f(\mathbf{x})\d{\mathbf{x}}\leq \Avg(f,\mathcal{D}_s^i),\quad i=0,1.$$
We complete the proof by dividing these inequalities by the first term and adding side by side, because $\partial h_{O_s}^{3/4}(\mathcal{D})=h_{O_s}^{3/4}(\mathcal{U}_{O_0})\cup h_{O_s}^{3/4}(\mathcal{U}_{O_1})$.	
\end{proof}	
\begin{corollary}
		Let $s^*=\frac{\sin\eta_0}{\sin\eta_0+\sin\eta_1}$. The equations   \eqref{eq:dipyramid_side_areas}, \eqref{eq:dipyramid_volume_split} and Theorem \ref{thm:dipyramid2} imply that 
		$$\Avg(f,\mathcal{D})\leq \tfrac{1}{4}f(O_{s^*})+\tfrac{3}{4}\Avg(f,\partial\mathcal{D}),$$
		$$\Avg(f,\partial h_{O_{s^*}}^{3/4}(\mathcal{D}))\leq \Avg(f,\mathcal{D}).$$
	\end{corollary}
\vspace{2em}	

With $n$ growing to infinity our dipyramids approximate a \textit{dicone}, where the $n$-gon $P$ gets replaced by a circle of radius $R$. All formulas \eqref{eq:dipyramidO_0}-\eqref{eq:dipyramidO_0O_1}, \eqref{neq:dipyramidO_s}-\eqref{neq:dipyramidO_s_inv2} remain valid, while the formula \eqref{eq:dipyramid_sticks} needs a modification.

Let us introduce a coordinate system in the most natural way (center $O$ at the center of $P$, $z$-axis along the line $l$ and $x$-axis along $OA_0$). Then we have
\begin{align*}
\Avg(f,\mathcal{S}_{O_0})	&=\frac{1}{n|O_0A_0|}	\sum_{k=0}^{n-1}\int_{O_0A_k} f(x,y,z)\d{l}\\
\intertext{with $x=r\cos\varphi,\ y=r\sin\varphi,\ z=(R-r)\cot\eta_0,\ 0\leq r\leq R$}
	&	=\frac{1}{nR}	\sum_{k=0}^{n-1}\int_{0}^R f\left(r\cos\tfrac{2\pi k}{n},r\sin\tfrac{2\pi k}{n},(R-r)\cot\eta_0\right)\d{r}\\
	&\to \frac{1}{2\pi R}\int_0^R\int_0^{2\pi} f\left(r\cos\varphi,r\sin\varphi,(R-r)\cot\eta_0\right)\d{\varphi}\d{r}\\
	&=\frac{1}{2\pi R}\iint\limits_{x^2+y^2\leq R^2} f\left(x,y,(R-\sqrt{x^2+y^2})\cot\eta_0\right)\frac{1}{\sqrt{x^2+y^2}}\d{x}\d{y}\\
	&=\frac{\sin\eta_0}{2\pi R}\iint\limits_{\mathcal{U}_{O_0}} \frac{f(x,y,z)}{\sqrt{x^2+y^2}}\d{S}.
\end{align*}	
Therefore the following result holds.
	\begin{theorem}
		If $f$ is a convex function defined on a dicone $\mathcal{D}$, and $g(\mathbf{x})=\frac{f(\mathbf{x})}{\dist(\mathbf{x},l)}$ ($\dist(\mathbf{x},l)$ denotes the distance from $\mathbf{x}$ to the line $l$),  then
		$$\Avg(f,\mathcal{D})\leq\frac{R}{4}\left(\Avg(g,\mathcal{U}_{O_0})+\Avg(g,\mathcal{U}_{O_1})\right).$$
	\end{theorem}

%% file: cube.tex
\section{Cube}

A cube being a Platonic body enjoys all properties  discussed in Section \ref{sec:Platonic}. In this section we present a handful of other applications of Theorems \ref{th:R} and \ref{th:L}. 

\begin{theorem}\label{thm:cube3}
	Let $\mathcal{C}$ be a cube, and $f\colon\mathcal{C}\to\RR$ be a function convex on every pyramid formed by a face and the center $O$ of the cube. Fix two opposite vertices of a cube and let $P$ be a set being the sum  of six diagonals of faces meeting at these vertices. Let $Q$ be the set of three main diagonals joining the remaining six vertices. Then  
	$$\Avg(f,\mathcal{C})\leq\frac{1}{2}\Avg(f,P)+\frac{1}{2}\Avg(f,Q).$$
\end{theorem}
\begin{proof}
	Let $ABCD$ be the face containing diagonal $AC$ in $P$. The pyramid $ABCDO$ is the sum of two simplices $ABCO$ and $ACDO$. Splitting vertices of each of them into   groups $\{AC\},\{BO\}$ and $\{AC\},\{DO\}$ respectively one gets
	\begin{align*}
		\Avg(f,ABCO)&\leq \frac{1}{2}\left(\Avg(f,AC)+\Avg(f,BO)\right)	\\
		\Avg(f,ACDO)&	\leq \frac{1}{2}\left(\Avg(f,AC)+\Avg(f,DO)\right)
	\end{align*}
	which gives
	$$2\Avg(f,ABCDO)\leq \Avg(f,AC)+\frac{1}{2}(\Avg(f,BO)+\Avg(f,DO)).$$
	We complete the proof in usual way, applying the same process to all diagonals in $P$.
\end{proof}
\begin{theorem}\label{thm:cube4}
Let $\mathcal{C}$ be a cube, and $f\colon\mathcal{C}\to\RR$ be a function convex on every pyramid formed by a face and the center $O$ of the cube.	Fix two opposite vertices of $\mathcal{C}$ and let $S$ be a set being the sum  of edges meeting at these vertices. Let $Q$ be the set of three main diagonals joining the remaining six vertices. Then
	$$\Avg(f,\mathcal{C})\leq\frac{1}{2}\Avg(f,S)+\frac{1}{2}\Avg(f,Q).$$
\end{theorem}
\begin{proof}
	The proof goes exactly the same way as the previous one, but this time we split the vertices of simplices into groups $\{CO\},\{AB\}$ and $\{CO\},\{AD\}$ respectively which leads to
	$$2\Avg(f,ABCDO)\leq \Avg(f,OC)+\frac{1}{2}(\Avg(f,AB)+\Avg(f,AD)).\qedhere $$
\end{proof}

The cube can be split into six simplices of equal volumes in different ways. One of them is particularly interesting - we shall call it a \textit{diagonal split}. Select two opposite vertices, say $O_1$ and $O_2$. The remaining vertices can be connected by edges of the cube so that they form a closed polygonal line $L=V_0\ldots V_5V_0$. The diagonal split consists of six simplices $O_1O_2V_kV_{k+1}$. Note that $O_1V_k$ and $O_2V_{k+1}$ are of the same length - they are both edges of the cube or diagonals of its faces. The next theorem shows how the diagonal split can be explored.
\begin{theorem}\label{thm:cube5}
	Consider a diagonal split of  a cube $\mathcal{C}$.  Denote by $S$ the set of six edges  adjacent to $O_1$ and $O_2$ and by  $P$ the set of six diagonals of faces  adjacent to $O_1$ and $O_2$. If $f\colon\mathcal{C}\to \RR$ is convex on each simplex of  the split, then
\begin{align}
	\Avg(f,\mathcal{C})&\leq \frac{1}{2}\Avg(f,O_1O_2)+\frac{1}{2}\Avg(f,L),\label{neq:cube5_1}	\\
		\Avg(f,\mathcal{C})&\leq\Avg(f,P)\label{neq:cube5_2},	\\
		\Avg(f,\mathcal{C})&\leq\Avg(f,S)\label{neq:cube5_3},	
\end{align}
(see Figure \ref{fig:HermiteHadamardInequalityInCube}).
\end{theorem}

\begin{proof}
	To prove \eqref{neq:cube5_1} use Theorem \ref{th:R} dividing the vertices of $O_1O_2V_kV_{k+1}$ into groups $\{O_1O_2\}$ and $\{V_kV_{k+1}\}$. Two other splits lead to \eqref{neq:cube5_2} and \eqref{neq:cube5_3}.
\end{proof}
\begin{figure}[htbp]
\begin{center}

\includegraphics{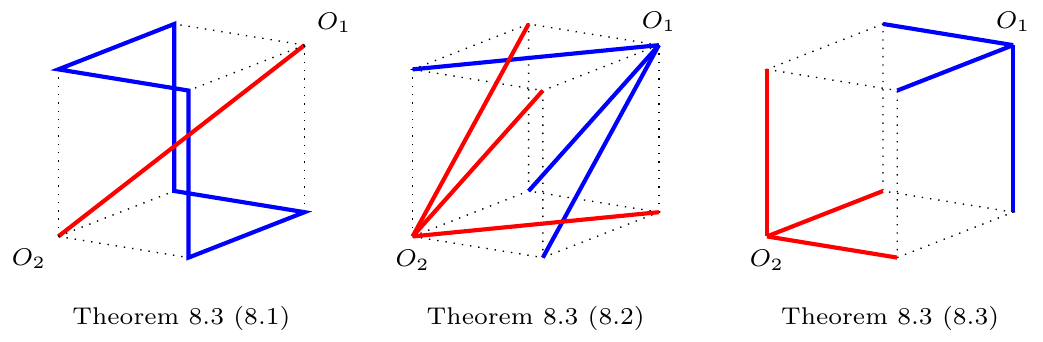}
	\caption{Hermite-Hadamard inequalities for cube}
	\label{fig:HermiteHadamardInequalityInCube}
\end{center}
\end{figure}

Next theorem presents an interesing asymmetric case:
\begin{theorem}\label{thm:cube6}
	Let $A$ be a vertex of the cube $\mathcal{C}$ and let $S$ be the set consisting of its faces nonadjacent to $A$. If $f:\mathcal{C}\to\RR$ is convex, then
	\begin{align}
		\Avg(f,\mathcal{C})&\leq \frac{1}{4}f(A)+\frac{3}{4}\Avg(f,S),\label{neq:cube6_1}	\\
		\Avg(f,h_A^{3/4}(S))&\leq \Avg(f,\mathcal{C}),\label{neq:cube6_2}
	\end{align}
	(see Figure \ref{fig:HermiteHadamardInequalityInCubeCtd}).
\end{theorem} 
We leave the obvious proof to the reader.
\begin{figure}[htbp]
\begin{center}
\includegraphics{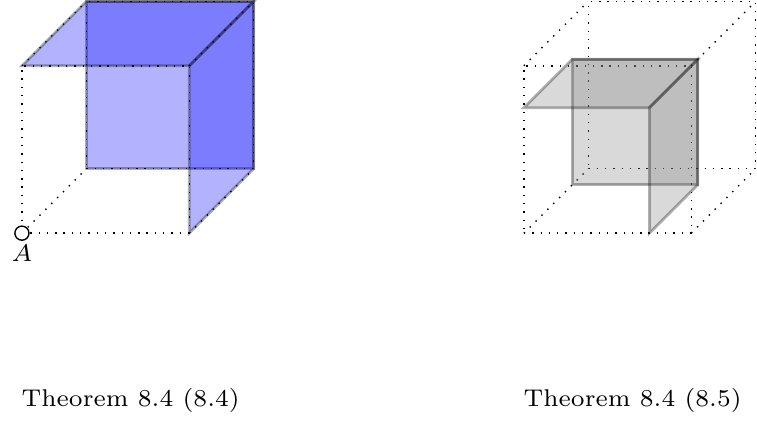}
	\caption{Hermite-Hadamard inequalities for cube ctd.}
	\label{fig:HermiteHadamardInequalityInCubeCtd}
\end{center}
\end{figure}